\documentclass[twoside,11pt]{article}

\usepackage[preprint]{jmlr2e}
\usepackage{algpseudocode}
\usepackage{algorithm}
\usepackage{float}
\usepackage{amsmath}
\usepackage{amssymb}
\usepackage{optidef}

\ShortHeadings{Lyapunov-IQC Framework for Uniform Stability}{Li and Daescu}
\firstpageno{1}

\begin{document}

\title{A Unified Lyapunov-IQC Framework for Uniform Stability of Smooth Quadratic First-Order Accelerated Optimizers}

\author{\name Don Li \email don@pdx.edu \\
       \addr Department of Mathematics \& Statistics\\
       Portland State University \\
       Portland, OR 97201, USA
       \AND
       \name Dacian Daescu \email daescu@pdx.edu \\
       \addr Department of Mathematics \& Statistics \\
       Portland State University \\
       Portland, OR 97201, USA}

\editor{N/A}

\maketitle

\begin{abstract}
We develop a unified Lyapunov--integral quadratic constraint (IQC) framework for establishing uniform stability of first-order accelerated optimization algorithms in the $\beta$-smooth and $\gamma$-strongly convex regime. Classical analyses of uniform stability, such as the work of Hardt, Recht, and Singer for stochastic gradient descent (SGD), rely on direct coupling arguments and case-by-case control of iterate differences under random sampling. Extending such arguments to accelerated methods, such as Nesterov Accelerated Gradient (NAG), is complicated by the presence of higher-order state dynamics induced by momentum. We first extend this classical approach with the use of Lyapunov functions to provide a uniform stability bound for smooth quadratic NAG, and supplement this result with small-scale numerical experiments. We then extend this framework by modeling first-order accelerated optimizers as Lur’e-type feedback interconnections between a linear dynamical system and a (non-linear) gradient operator. $\beta$-Smoothness and $\gamma$-strong convexity are encoded a sector IQC inequality. Under this representation, uniform stability is certified via the existence of a quadratic Lyapunov function satisfying a finite-dimensional linear matrix inequality (LMI) in the form of a feasibility problem, which can be solved via semi-definite programming (SDP). We instantiate this framework for NAG and show how classical uniform stability bounds can be recovered via this framework. These results underscore a structural connection between optimization dynamics and robust control theory, providing a modular methodology for reliable and reproducible numerical certification of uniform stability and generalization behavior of first-order methods via convex optimization tools that is adaptable to increasingly complex optimization algorithms.
\end{abstract}

\section{Introduction}

Despite their efficacy with respect to solving complex, large-scale optimization problems, all convex optimization algorithms are susceptible to \textit{overfitting}, i.e., being unable to generalize beyond its training dataset to new, unseen data. This is a major challenge for the application of said convex optimization algorithms to machine learning in practice. Such application of convex optimization algorithms to machine learning is framed in terms of the \textit{empirical risk minimization} (ERM) paradigm, which is represented by the unconstrained minimization problem

\begin{equation}
\underset{x \in \mathbb{R}^d}{\text{min}} \quad f(x) = \frac{1}{n}\sum_{i=1}^{n} f_i(x),
\end{equation}

\noindent
where $f_i(x)$ is the \textit{loss function} with respect to parameter $w$ at sample $z_i \in S = (z_1, ...,z_n)$. The core task of any optimization algorithm is to minimize $f(x)$ in the form of (1) by updating from $x_t$ to each successive $x_{t+1}$ at each iteration $t$. First-order optimization algorithms are the class of optimization algorithms that achieve this by using the gradient of the loss function, $\nabla f (x_t)$, as its \textit{oracle} to perform parameter updates. The most fundamental first-order optimization algorithm is (stochastic) gradient descent, whose update rule is

\begin{equation}
x_{t+1} = x_t - \eta \nabla f(x_t), 
\end{equation}

\noindent
where $\eta$ denotes the learning rate hyperparameter, which governs the magnitude of the parameter update via the gradient oracle. \\
\\
It is a basic result in convex analysis that any strongly-convex function is guaranteed to have a global minimum, and our analysis here, while focused on generalization behavior, considers only strongly convex loss functions. We detail the relevant convexity properties in the following subsection.

\subsection{Convexity Properties of the Smooth Quadratic Regime}

We first formally define $\gamma$-strong convexity:

\begin{definition}[$\gamma$-Strong Convexity]
The loss $f(w)$ is $\gamma$-strongly convex in $w$, i.e., $\forall w, w'$, there exists $\gamma \in \mathbb{R}$ such that

\begin{center}
$f(w) \geq f(w) + \nabla (w)^T (w' - w) + \frac{\gamma}{2}||w' - w||^2$.
\end{center}
\end{definition}

\noindent
The strong convexity condition guarantees that the minimum curvature of the loss function is at least quadratic, the magnitude of which is governed by $\gamma$. Strong convexity of the loss function implies monotonicty of the gradient, i.e.,

\begin{corollary}[Gradient Monotonicty]
Let $f$ be $\gamma$-strongly convex. Then, $\forall x,y \in \text{dom}(f)$,

\begin{center}
$\langle \nabla f(x) - \nabla f(y), x-y \rangle \geq \gamma || x-y ||^2$.
\end{center}
\end{corollary}

\noindent
We likewise assume $\beta$-Lipschitzness of the gradient:

\begin{definition}[$\beta$-Smoothness]
The loss $f(w)$ has \textbf{$\beta$-Lipschitz gradients}, or is \textbf{$\beta$-smooth}, if $\exists \beta \in \mathbb{R}$ such that $\forall w, w'$,

\begin{center}
$||\nabla f(w) - \nabla f(w')|| \leq \beta ||w - w'||$.
\end{center}
\end{definition}

\noindent
$\beta$-smoothness ensures that the ``growth" of the gradient is controlled by $\beta$. Consequently, $\beta$ acts as the upper bound of the curvature of the loss function. We likewise assume boundedness of the gradient:

\begin{definition}[Bounded Gradients]
The loss $f(w)$ has \textbf{bounded gradients}, if, $\forall w$, there exists $G \in \mathbb{R}$ such that

\begin{center}
$||\nabla f(w)|| \leq G$.
\end{center}
\end{definition}

\noindent
Since the gradient is uniformly bounded, i.e., $||\nabla f(w)|| \leq G$, we have the following corollary:

\begin{corollary}[$G$-Lipschitz Loss Function]
Let $f(w)$ denote the loss function, where $f(w)$ is differentiable with respect to $w$. Since the gradient is uniformly bounded where  $||\nabla f(w)|| \leq G$ for all $w$, it follows that $|f(w) - f(w')| \leq G||w - w'||$.
\end{corollary}

\noindent
As shorthand, we refer to these convexity conditions collectively as the \textit{smooth quadratic regime} hereafter, and these convexity conditions are assumed for all following results in this paper.

\subsection{Nesterov Accelerated Gradient}

While our Lyapunov-IQC framework concerns a broader class of first-order optimization algorithms, we motivate the development of and apply this framework to Nesterov Accelerated Gradient (NAG) in particular. Building upon the first accelerated optimizer (gradient descent with classical momentum, or the ``heavy-ball" method, HB) introduced by \cite{Pol64}, Nesterov Accelerated Gradient (NAG) was introduced in \cite{Nes83}, attaining, in the smooth general convex regime, a convergence rate $O(\frac{1}{T^2})$, an improvement over the $O(\frac{1}{T})$ convergence rate achieved by SGD in the same regime. Accelerated optimizers such as NAG achieve such improved convergence rates via incorporating a ``momentum" term in the parameter update. Doing so produces two update rules per iteration,

\begin{equation}
v_{t+1} = \mu v_t - \eta \nabla f(w_t + \mu v_t),
\end{equation}

\begin{equation}
w_{t+1} = w_t + v_{t+1},
\end{equation}

\noindent
where $\mu \in [0,1)$ is the momentum parameter and $v_t$ is the momentum term. As \cite{Bub15} notes, the NAG update rules, specifically in the smooth quadratic regime, where, for $\kappa = \frac{\beta}{\gamma}$ (where $\kappa$ is the \textit{condition number}), can be formulated as

\begin{equation}
v_{t+1} = w_t - \frac{1}{\beta} \nabla f (w_t, z_t),
\end{equation}

\begin{equation}
w_{t+1} = (1 + \frac{\sqrt{\kappa}-1}{\sqrt{\kappa}+1})v_{t+1} - (\frac{\sqrt{\kappa}-1}{\sqrt{\kappa}+1})v_t.
\end{equation}

\noindent
Let $\theta := \frac{\sqrt{\kappa}-1}{\sqrt{\kappa}+1}$. Then equations (5) and (6) can likewise be written as

\begin{equation}
v_{t+1} = w_t - \frac{1}{\beta} \nabla f(w_t),
\end{equation}

\begin{equation}
w_{t+1} = (1 + \theta)v_{t+1} - \theta v_t.
\end{equation}

\noindent
The execution of a single epoch of NAG (over $T \in \mathbb{N}$ iterations) is detailed in the following pseudocode:

\begin{algorithm}[H]
\caption{Nesterov Accelerated Gradient (NAG)}\label{alg:cap}
\begin{algorithmic}

\Require Initialization of $w_t$ and $v_t$ at $t=1$ for $t \in \mathbb{N}$, $\eta$, and $\mu \in [0,1)$

\While{$t \leq T$}
    \State Compute $w_t + \mu v_t$
    \Comment{The ``look-ahead" position}

    \State Compute $\nabla \ell (w_t + \mu v_t)$
    \Comment{Gradient at look-ahead position}

    \State $v_{t+1} = \mu v_t - \eta \nabla \ell(w_t + \mu v_t)$
    \State $w_{t+1} = w_t + v_{t+1}$
    \State $t += 1$
\EndWhile

\noindent
\Return $w_T$
\end{algorithmic}
\end{algorithm}

\subsection{Algorithmic (Uniform) Stability}

We analyze the generalization behavior of an optimization algorithm in terms of \textit{uniform stability} as developed by \cite{BE02}, which is defined formally as follows:

\begin{definition}[$\epsilon_n$-Uniform Stability]
Let $A$ be a particular stochastic optimization algorithm ran over $n$ training iterations, let $S$ be a dataset where $S'$ is $S$ with a single data point $z_j$ modified, and let $\xi = \{z_1, z_2, ... , z_n\}$ be a coupled random sampling of indices of $S$ and $S'$. Moreover, consider a loss function for $S$, $\ell(A(S, \xi), z)$, and the same loss function for $S'$, $\ell(A(S', \xi), z)$. Let $\epsilon_n$ denote the (uniform) stability parameter of $A$, i.e., the smallest value such that replacing any single example in the training set (here, $z_*$) changes the loss on any test point by at most $\epsilon_n$. Then $A$ is \textbf{uniformly stable} iff

\begin{center}
$\text{sup}_{S,S', z} \mathbb{E}[|\ell(A(S, \xi), z) - \ell(A(S', \xi), z)|] \leq \epsilon_n$.
\end{center}

\noindent
Moreover, we say that $A$ is stable iff $\epsilon_n = O(\frac{1}{n})$.
\end{definition}

\noindent
Informally, an algorithm is uniformly stable if, assuming a uniform distribution on a dataset $S$ where $S'$ is $S$ modified at a single data point $z_j$, the trajectories of the loss functions on each of $S$ and $S'$ stay within $\epsilon$ of one another. In this sense, uniform stability is a metric of how robust an optimization algorithm is to small changes in its input(s).

\subsection{Lyapunov Functions}

\cite{HRS16} derived an $O(\frac{2G^2}{\gamma n})$ uniform stability bound for smooth quadratic SGD (cf. Appendix A for their proof in comprehensive form). Recall that SGD is governed by the single update rule $w_{t+1} = w_t - \eta \nabla f(w_t)$. Hence, for SGD, there is only the parameter difference $w_t - w_t'$ to track, where $w_t$ is the parameter value for the optimizer ran on $S$ and $w_t'$ is the parameter value for the optimizer ran on $S'$. The crux of the proof from \cite{HRS16} for the uniform stability of smooth quadratic SGD is to find the expected value of the parameter difference in the case where the differing sample $z_j$ is not sampled (i.e., $i_t \neq j$), then to do so for the case where $z_j$ is sampled ($i_t = j$), and then to find the expected value of the parameter difference as a weighted average of both via the law of total expectation from probability theory. Let $A$ denote the event where the differing sample is not selected, so $P(A) = 1 - \frac{1}{n}$, so $A^c$ is the event where the differing sample is selected so $P(A^c) = \frac{1}{n}$. Let $\mathbb{E}[\delta_{t+1}] = \mathbb{E}[w_{t+1} - w'_{t+1}]$. Then by the law of total expectation we have

\begin{equation}
\mathbb{E}[\delta_{t+1}] = (1-\frac{1}{n})\mathbb{E}[||w_{t+1} - w'_{t+1}|| | i_t \neq j] + \frac{1}{n}\mathbb{E}[||w_{t+1} - w'_{t+1}|| | i_t = j].
\end{equation}

\noindent
Refer to the full proof in Appendix A for the missing details. Eqn. (9), along with $G$-Lipschitzness of the loss function, yields the $O(\frac{2G^2}{\gamma n})$ uniform stability bound as derived by \cite{HRS16}. \\
\\
\cite{CJY18} derived an $O(\frac{4\beta^2}{\gamma n}[1-(1 - \frac{1}{\sqrt{\kappa}})^T])$ uniform stability bound for smooth quadratic NAG. Recall that NAG has two update rules, one for ``velocity" $v_t$ and one for the parameter value $w_t$. Hence, case by case iterate difference analysis, as \cite{HRS16} did for SGD, is made cumbersome by having two iterate differences to track simultaneously: $\Delta w_t = w_t - w_t'$ and $\Delta v_t = v_t - v_t'$. We sidestep the complications made by these higher-order dynamics via \textit{Lyapunov functions}. Lyapunov functions have long been used in the theory of dynamical systems for proving stability results for (non-linear) dynamical systems, with the highly useful feature that doing so does not require explicitly solving ODE's/PDE's. Moreover, for our purposes, a Lyapunov function $\mathcal{V}(x)$ has the additional useful property of taking vector/matrix inputs and producing scalar output, i.e., $\mathcal{V}: \mathbb{R}^{m \times n} \rightarrow \mathbb{R}$. Recall that NAG has two update rules per iteration, whereas SGD only has one, i.e., with $\theta := \frac{\sqrt{\kappa}-1}{\sqrt{\kappa}+1}$ for $\kappa = \frac{\beta}{\gamma}$, 

\begin{equation}
v_{t+1} = w_t - \frac{1}{\beta} \nabla f(w_t),
\end{equation}

\begin{equation}
w_{t+1} = (1 + \theta)v_{t+1} - \theta v_t.
\end{equation}

\noindent
We consider the state vector $x_t := \begin{bmatrix}
    \Delta w_t \\
    \Delta v_t
\end{bmatrix}$, where $\Delta w_t := w_t - w_t'$ and $\Delta v_t := v_t - v_t'$ (where $w_t, v_t$ are the parameter iterate and velocity iterates from the original dataset $S$ and $w'_t, v'_t$ are the parameter iterate and velocity iterates from the perturbed dataset $S'$). Thus, $x_t = \begin{bmatrix}
    \Delta w_t \\
    \Delta v_t
\end{bmatrix} = \begin{bmatrix}
    w_t - w'_t \\
    v_t - v'_t
\end{bmatrix}$. $x_t$ is the input of the Lyapunov function. In particular, we use a quadratic Lyapunov function, i.e.,

\begin{definition}[Quadratic Lyapunov Function]
Consider a discrete linear time-invariant (LTI) system $\dot{x} = Ax$, where $x \in \mathbb{R}^n$ and $A \in \mathbb{R}^{n \times n}$, with $x^* = 0$ as the fixed point of the LTI. A \textbf{quadratic Lyapunov function} of such an LTI is a scalar function of the form

\begin{equation}
\mathcal{V}(x) = x^T Px,
\end{equation}

\noindent
where $P \in \mathbb{R}^{n \times n}$ is a symmetric positive definite matrix, i.e., $P = P^T \succ 0$. \\
\\
Moreover, $\mathcal{V}(x)$ is a valid Lyapunov function iff the following two conditions hold:

\begin{enumerate}
    \item \textbf{Positive Definiteness} $\forall x \neq 0$, $\mathcal{V}(x) = x^T Px > 0$, and $\mathcal{V}(0)=0$,

    \item \textbf{Negative Semi-Definiteness of the Derivative} $\forall x \neq 0$, $\dot{V}(x) \leq 0$.
\end{enumerate}
\end{definition}

\noindent
For discrete LTI systems, the matrix $P$ can be obtained by solving the discrete-time Lyapunov equation

\begin{equation}
A^TPA - P = -Q
\end{equation}

\noindent
for some $Q \succ 0$. Moreover, we must have $P=P^T \succ 0$ such that $A^TPA - P \prec 0$. This implies the following corollary:

\begin{corollary}
Let $\mathcal{V}(x) = x^T Px$. If $P = P^T \succ 0$ such that $A^TPA - P = -Q$ for $Q \succ 0$, then $\mathcal{V}(x) = x^T Px$ is a valid Lyapunov function.
\end{corollary}

\noindent
Consider $\Delta v_{t+1} = v_{t+1} - v'_{t+1}$. This implies $\Delta v_{t+1} = \Delta w_t - \frac{1}{\beta}(\nabla f(w_t,z_t) - \nabla f (w'_t, z'_t))$. If $\Delta g_t := \nabla f(w_t,z_t) - \nabla f(w'_t, z'_t)$, then $\Delta v_{t+1} = \Delta w_t - \frac{1}{\beta}\Delta g_t$. We next consider $\Delta w_{t+1} = w_{t+1} - w'_{t+1} = (1+\theta)\Delta v_{t+1} - \theta \Delta v_t = (1+\theta)\Delta w_t - \theta \Delta v_t - (\frac{1+\theta}{\beta})\Delta g_t$. Thus, the subsequent state vector can be written as

\begin{equation}
x_{t+1} = \begin{bmatrix}
    \Delta w_{t+1} \\
    \Delta v_{t+1}
\end{bmatrix} = \begin{bmatrix}
    (1+\theta)\Delta w_t - \theta \Delta v_t - (\frac{1+\theta}{\beta})\Delta g_t \\
    \Delta w_t - \frac{1}{\beta} \Delta g_t
\end{bmatrix}.
\end{equation}

\noindent
Before we construct our quadratic Lyapunov function, we take advantage of the properties of the smooth quadratic regime to derive some useful results, particularly with respect to the gradient difference $\Delta g_t$. To do so, consider the following lemma from \cite{NW06}:

\begin{lemma}[Integral Form of Taylor's Theorem]
Let $f: \mathbb{R}^d \rightarrow \mathbb{R}$ be continuously differentiable and let $\nabla f$ be Lipschitz-continuous. Then for any $x,y \in \mathbb{R}^d$,

\begin{equation}
\nabla f(x) - \nabla f(y) = (\int_{0}^{1} \nabla^2 f (y + t(x-y))dt)(x-y). 
\end{equation}
\end{lemma}

\begin{proof}
Define the affine function $\phi(t) = \nabla f(y + t(x-y))$ for $t \in [0,1]$. Then $\phi(0) = \nabla f(y)$ and $\phi(1) = \nabla f(x)$. This implies $\phi(1) - \phi(0) = \nabla f(x) - \nabla f(y)$. We compute $\phi'(t)$. By the chain rule, $\phi'(t) = \nabla^2 f(y + t(x-y))(x-y)$. Moreover, by the Fundamental Theorem of Calculus, $\phi(1) - \phi(0) = \int_{0}^{1} \phi'(t)dt$. Then substituting $\phi(1) - \phi(0) = \nabla f(x) - \nabla f(y)$ and $\phi'(t) = \nabla^2 f(y + t(x-y))(x-y)$ as we derived yields

\begin{center}
$\nabla f(x) - \nabla f(y) = (\int_{0}^{1} \nabla^2 f (y + t(x-y))dt)(x-y)$.
\end{center}
\end{proof}

\noindent
We use this lemma from \cite{NW06} to derive the following relation and bound on $||H_t||$ and $||\Delta g_t||$:

\begin{theorem}
Let $f(w_t, z_t)$, $f (w_t', z_t')$ where $z_t = z'_t$ be twice differentiable as well as $\gamma$-strongly convex and $\beta$-smooth. Then, for $H_t := \int_{0}^{1} \nabla^2 \ell (w_t' - s(w_t - w_t'), z_t)ds$, 

\begin{center}
$\Delta g_t = H_t \Delta w_t$,
\end{center}

\noindent
with $||H_t|| \leq \beta$ and $||\Delta g_t|| \leq \beta ||\Delta w_t||$.
\end{theorem}

\begin{proof}
Define the affine function $\phi(s) := \nabla \ell (w'_t + s(w_t - w'_t), z_t)$ for $s \in [0,1]$. By the Integral Form of Taylor's Theorem, this implies

\begin{center}
$\nabla \ell (w_t, z_t) - \nabla \ell (w'_t, z'_t) = (\int_{0}^{1} \nabla^2 \ell (w_t' + s(w_t - w_t'), z_t)ds)(w_t - w'_t)$.
\end{center}

\noindent
But we already define $\Delta g_t := \nabla f(w_t, z_t) - \nabla f(w'_t, z'_t)$ and $\Delta w_t := w_t - w'_t$. Moreover, let $H_t := \int_{0}^{1} \nabla^2 \ell (w_t' + s(w_t - w_t'), z_t)ds$. Then

\begin{center}
$\Delta g_t = H_t \Delta w_t$.
\end{center}

\noindent
We next show that $||H_t|| \leq \beta$. Since $f(w_t, z_t)$ is twice-differentiable, $\gamma$-strongly convex, and $\beta$-smooth, it follows that $\gamma I \preceq \nabla^2 f(w_t, z_t) \preceq \beta I$. We know that $\nabla^2 \ell (w_t, z_t)$ and $\beta I$ are both symmetric matrices since $\ell (w_t, z_t)$ is twice-differentiable so all of the second-order derivatives contained in $\nabla^2 \ell(w_t, z_t)$ are continuous. Since L\"{o}wner order is preserved under integration for any two symmetric matrices $A(x)$, $B(x)$, where $A(x) \preceq B(x)$ (i.e., where $B(x) - A(x) \succeq 0$), it follows that $\int_{0}^{1} \nabla^2 \ell (w'_t + s(w_t - w'_t), z_t)ds \preceq \beta I$ for $s \in [0,1]$. But recall $H_t = \int_{0}^{1} \nabla^2 \ell (w_t' + s(w_t - w_t'), z_t)ds$. This implies $H_t \preceq \beta I$ with $H_t = H^T_t$. The norm of any symmetric matrix is equal to the absolute value of its largest eigenvalue. Let $\lambda_{\text{max}(H_t)}$ denote maximum eigenvalue of $H_t$ (so $\lambda_{\text{max}(H_t)} = ||H_t||$). Since $H_t \preceq \beta I$, it follows that $\lambda_{\text{max}(H_t)} \leq \beta$, which implies $||H_t|| \leq \beta$. \\
\\
We last show $||\Delta g_t || \leq \beta ||\Delta w_t||$. $f(w_t, z_t)$ is assumed to be $\beta$-smooth, i.e., $||\nabla f(w_t,z_t) - \nabla f(w'_t,z'_t)|| \leq \beta ||w_t - w'_t||$. But recall we define $\Delta g_t := \nabla f(w_t,z_t) - \nabla f(w'_t,z'_t)$ and $\Delta w_t := w_t - w'_t$. Then $||\Delta g_t|| \leq \beta ||\Delta w_t||$.
\end{proof} 

\noindent
Recall that the subsequent state vector can be written as

\begin{equation}
x_{t+1} = \begin{bmatrix}
    \Delta w_{t+1} \\
    \Delta v_{t+1}
\end{bmatrix} = \begin{bmatrix}
    (1+\theta)\Delta w_t - \theta \Delta v_t - (\frac{1+\theta}{\beta})\Delta g_t \\
    \Delta w_t - \frac{1}{\beta} \Delta g_t
\end{bmatrix}.
\end{equation}

\noindent
But we derived from Theorem 10 that $\Delta g_t = H_t\Delta w_t$ for $H_t = \int_{0}^{1} \nabla^2 \ell (w_t' + s(w_t - w_t'), z_t)ds$ for some affine function $\phi(s) = \nabla \ell (w'_t + s(w_t - w'_t), z_t)$ over $s \in [0,1]$. Thus, the subsequent state vector can be written as

\begin{equation}
x_{t+1} = \begin{bmatrix}
    \Delta w_{t+1} \\
    \Delta v_{t+1}
\end{bmatrix} = \begin{bmatrix}
    (1+\theta)\Delta w_t - \theta \Delta v_t - (\frac{1+\theta}{\beta})H_t\Delta w_t \\
    \Delta w_t - (\frac{1}{\beta})H_t\Delta w_t
\end{bmatrix}.
\end{equation}

\noindent
Rearranging this matrix equation to express $x_{t+1}$ as a product of $x_t$ yields

\begin{equation}
x_{t+1} = \begin{bmatrix}
    (1+\theta)(I-\frac{H_t}{\beta}) & -\theta I \\
    I - \frac{H_t}{\beta} & 0
\end{bmatrix}\begin{bmatrix}
    \Delta w_t \\
    \Delta v_t
\end{bmatrix}.
\end{equation}

\noindent
But recall $x_t = \begin{bmatrix}
    \Delta w_t \\
    \Delta v_t
\end{bmatrix}$. Thus, the iteration evolution of the state vector is governed by $x_{t+1} = A(H_t) x_t$, where $A(H_t) = \begin{bmatrix}
    (1+\theta)(I-\frac{H_t}{\beta}) & -\theta I \\
    I - \frac{H_t}{\beta} & 0    
\end{bmatrix}$. Moreover, let $\alpha := I - \frac{H_t}{\beta}$. Since $\gamma I \preceq H_t \preceq \beta I$, it follows that $0 \preceq \alpha \preceq (1 - \frac{\gamma}{\beta})I$. Since $f(w_t)$ is twice-differentiable, it follows that $H_t$ is symmetric, which in turn implies that $H_t$ is diagonalizable. Let $\lambda$ denote any arbitrary eigenvalue of $H_t$. Then $\gamma \leq \lambda \leq \beta$ since $\gamma I \preceq H_t \preceq \beta I$. Fix an arbitrary eigen-direction of $A(H_t)$. Then $\alpha = 1 - \frac{\lambda}{\beta}$ and $0 \leq \alpha \leq 1 - \frac{\gamma}{\beta}$. Thus, the iteration evolution of the state vector for fixed eigen-direction of $H_t$ is $x_{t+1} = A_{\alpha}x_t$, where $A_\alpha = \begin{bmatrix}
    (1+\theta)]\alpha & -\theta \\
    \alpha & 0
\end{bmatrix}$.

\section{Direct Construction of NAG Quadratic Lyapunov Function}

Here, we explicitly construct the quadratic Lyapunov function used to prove the uniform stability of smooth quadratic NAG for the case of scalar $\alpha$ where $\alpha \in [0, 1 - \frac{\gamma}{\beta}]$.

\begin{theorem}[Lyapunov Function for NAG Uniform Stability]
Assume that the loss $f(w)$ is $\gamma$-strongly convex with $\beta$-Lipschitz ($\beta$-smooth) and $G$-bounded gradients for all $z$. Run NAG with constant step size $\eta \leq \frac{1}{\beta}$ and fixed momentum parameter $0 \leq \mu < 1$ for $T$ iterations on a dataset $S$ and $S'$, the latter of which is $S$ with a single arbitrary data point $z_{j_t}$ modified. Let $x_t = \begin{bmatrix}
    \Delta w_t \\
    \Delta v_t
\end{bmatrix}$  denote the state vector which contains $\Delta w_t = w_t - w'_t$ and $\Delta v_t = v_t - v'_t$. Then we obtain the valid quadratic Lyapunov function 

\begin{equation}
\mathcal{V}_{\epsilon}(x_{t}) = a||\Delta w_{t}||^2 + b ||\Delta v_{t}||^2 + 2c\langle \Delta w_{t}, \Delta v_{t} \rangle, 
\end{equation}

\noindent
where $a = 1$, $b=(1+\theta)^2 + \epsilon$, and $c = -(1+\theta)$, $\forall \epsilon > 0$.
\end{theorem}

\begin{proof}
We wish to find a quadratic Lyapunov function $\mathcal{V}(x)$ for NAG. By definition, such a quadratic Lyapunov function will take the form $\mathcal{V}(x) = x^T Px$. Recall the update rules for NAG from \cite{Bub15},

\begin{equation}
v_{t+1} = w_t - \frac{1}{\beta} \nabla \ell (w_t, z_t),
\end{equation}

\begin{equation}
w_{t+1} = (1 + \frac{\sqrt{\kappa}-1}{\sqrt{\kappa}+1})v_{t+1} - (\frac{\sqrt{\kappa}-1}{\sqrt{\kappa}+1})v_t.
\end{equation}

\noindent
$\kappa = \frac{\beta}{\gamma}$ is the condition number. Let $\theta = \frac{\sqrt{\kappa}-1}{\sqrt{\kappa}+1}$. Then the NAG update rules can be re-written as

\begin{equation}
v_{t+1} = w_t - \frac{1}{\beta} \nabla \ell (w_t, z_t),
\end{equation}

\begin{equation}
w_{t+1} = (1 + \theta)v_{t+1} - \theta v_t.
\end{equation}

\noindent
We next find the Lyapunov $\mathcal{V}(x_t) = x_{t}^T Px_t$ where $P = P^T \succ0$. Let $P = \begin{bmatrix}
    a & c \\
    c & b
\end{bmatrix}$. Clearly $P = P^T$. Consider $\mathcal{V}(x_{t+1}) = x_{t+1}^T P x_{t+1}$. Hence we have

\begin{equation}
\mathcal{V}(x_{t+1}) = x_{t+1}^T P x_{t+1} = \begin{bmatrix}
    \Delta w_{t+1}^T &  \Delta v_{t+1}^T
\end{bmatrix} \begin{bmatrix}
    a & c \\
    c & b
\end{bmatrix} \begin{bmatrix}
    \Delta w_{t+1} \\
    \Delta v_{t+1}
\end{bmatrix}.
\end{equation}

\noindent
Performing the needed matrix multiplication yields

\begin{equation}
\mathcal{V}(x_{t+1}) = a \Delta w_{t+1} \Delta w^T_{t+1} + b \Delta v_{t+1} \Delta v^T_{t+1} + c \Delta w^T_{t+1} \Delta v_{t+1} + c \Delta v_{t+1}^T \Delta w_{t+1}.
\end{equation}

\noindent
For any vector $u$, $||u||^2 = u^T u$. Moreover, by the definition of the dot product, $c \Delta w^T_{t+1} \Delta v_{t+1} + c \Delta v_{t+1}^T \Delta w_{t+1} = 2c \langle \Delta w_{t+1}, \Delta v_{t+1} \rangle$. Thus,

\begin{equation}
\mathcal{V}(x_{t+1}) = a||\Delta w_{t+1}||^2 + b||\Delta v_{t+1}||^2 + 2c\langle \Delta w_{t+1}, \Delta v_{t+1} \rangle.
\end{equation}

\noindent
A valid Lyapunov function must be non-increasing, i.e., $\mathcal{V}(x_{t+1}) - \mathcal{V}(x_t) \leq 0$. Recall $x_{t+1} = A_{\alpha}x_t$ for $A_{\alpha} = \begin{bmatrix}
    (1+\theta)\alpha & -\theta \\
    \alpha & 0
\end{bmatrix}$ with $\alpha \in [0, 1-\frac{\gamma}{\beta}]$, so we have $\mathcal{V}(x_t) = x_t^TPx_t$ and $\mathcal{V}(x_{t+1}) = x_t^TA_{\alpha}^TPA_{\alpha}x_t$. This implies

\begin{equation}
x_t^T(A_{\alpha}^TPA_{\alpha}-P)x_t \leq 0.
\end{equation}

\noindent
This implies the matrix inequality

\begin{equation}
A_{\alpha}^TPA_{\alpha}-P \preceq 0,
\end{equation}

\noindent
or $A_{\alpha}^TPA_{\alpha} \preceq P$. But this inequality does not guarantee the desired contraction from each $\mathcal{V}(x_t)$ to $\mathcal{V}(x_{t+1})$. Instead, to guarantee strict contraction on the Lyapunov function, we enforce the condition

\begin{equation}
\mathcal{V}(x_{t+1}) \leq (1-\rho)\mathcal{V}(x_t)
\end{equation}

\noindent
for $\rho > 0$. This is analogous to what \cite{HRS16} did in their uniform stability proof of SGD, whereby they derived $||w_{t+1} - w'_{t+1}|| \leq (1- \frac{\gamma}{\eta}(\frac{n-1}{n}))||w_t - w_t'||$, i.e., the parameter iterate divergence in SGD ran on $S$ and $S'$ strictly contracts by a factor of $1-\rho$ where $\rho = \frac{\gamma}{\eta}(\frac{n-1}{n})$ (cf. Appendix A). Enforcing this strict contraction condition on our Lyapunov function yields the matrix inequality

\begin{equation}
A_{\alpha}^TPA_{\alpha} \preceq (1-\rho)P,
\end{equation}

\noindent
so we wish to find $P = P^T = \begin{bmatrix}
    a & c \\
    c & b
\end{bmatrix}$ that satisfies the above matrix inequality. This matrix inequality can likewise be written as $A_{\alpha}^TPA_{\alpha}-(1-\rho)P \preceq 0$. Let $M_{\alpha} := A_{\alpha}^TPA_{\alpha}-(1-\rho)P$. Then by computation we obtain

\begin{equation}
M_{\alpha} = \begin{bmatrix}
    \alpha^2[a(1+\theta)^2+2c(1+\theta)+b]-(1-\rho)a & -\alpha\theta(a(1+\theta)+c)-(1-\rho)c \\
    -\alpha\theta(a(1+\theta)+c)-(1-\rho)c & a\theta^2-(1-\rho)b
\end{bmatrix} \preceq 0,
\end{equation}

\noindent
for $\alpha \in [0, 1 - \frac{\gamma}{\beta}]$. Note $M_{\alpha} = M_{\alpha}^T$. Let $S:= a(1+\theta)^2+2c(1+\theta)+b$ and $T := a(1+\theta)+c$. Then we have

\begin{equation}
M_{\alpha} = \begin{bmatrix}
    \alpha^2S-(1-\rho)a & -a\theta T - (1-\rho)c \\
    a\theta T - (1-\rho)c & a\theta^2-(1-\rho)b
\end{bmatrix}.
\end{equation}

\noindent
To ensure $M_{\alpha} \preceq 0$, we want $\alpha^2S-(1-\rho)a + a\theta^2-(1-\rho)b \leq 0$ (i.e., $\text{tr}(M_{\alpha}) \leq 0$) and $(\alpha^2S-(1-\rho)a)(a\theta^2-(1-\rho)b) - (a\theta T - (1-\rho)c)^2 \geq 0$ (i.e., $\text{det}(M_{\alpha}) \geq 0)$. By virtue of the invariance of the Lyapunov inequality under positive scaling of $P$, we normalize by setting $a=1$. Moreover, by construction, let $T=0$, which implies $c = -(1+\theta)$. Thus, we now have

\begin{equation}
M_{\alpha} = \begin{bmatrix}
    \alpha^2S-(1-\rho) & (1-\rho)(1+\theta) \\
    (1-\rho)(1+\theta) & \theta^2-(1-\rho)b
\end{bmatrix}.
\end{equation}

\noindent
Again by construction let $S=0$. This implies, with $a=1$ and $c=-(1+\theta)$, $(1+\theta)^2 - 2(1+\theta)^2+b = 0$, which implies $b = (1+\theta)^2$. \\
\\
Thus, we have $P = \begin{bmatrix}
    a & c \\
    c & b
\end{bmatrix} = \begin{bmatrix}
    1 & -(1+\theta) \\
    -(1+\theta) & (1+\theta)^2
\end{bmatrix}$. Note $\text{tr}(P) = 1 + (1+\theta)^2 \geq 0$ and $\text{det}(P) = (1+\theta)^2 -[(-(1+\theta))(-(1+\theta))] = 0 \geq 0$. One can verify that the eigenvalues of $P$ are $\lambda_1 = 0 \geq 0$ and $\lambda_2 = 1 + (1+\theta)^2 \geq 0$, so $P \succeq 0$. \\
\\
With $a=1$, $b = (1+\theta)^2$, and $c = -(1+\theta)$, our Lyapunov function is 

\begin{equation}
\mathcal{V}(x_{t+1}) = ||\Delta w_{t+1}||^2 + (1+\theta)^2||\Delta v_{t+1}||^2 -2(1+\theta)\langle \Delta w_{t+1}, \Delta v_{t+1}\rangle.
\end{equation}

\noindent
By completing the square, we likewise have

\begin{equation}
\mathcal{V}(x_{t+1}) = ||\Delta w_{t+1} - (1+\theta)\Delta v_{t+1}||^2.
\end{equation}

\noindent
But $\mathcal{V}(x_{t+1}) = ||\Delta w_{t+1} - (1+\theta)\Delta v_{t+1}||^2$ does not suffice as a valid Lyapunov function since $P = \begin{bmatrix}
    a & c \\
    c & b
\end{bmatrix} = \begin{bmatrix}
    1 & -(1+\theta) \\
    -(1+\theta) & (1+\theta)^2
\end{bmatrix} \succeq 0$, i.e., $P$ is a positive semi-definite but not positive definite matrix as required (recall that one of the eigenvalues of our originally constructed $P$ is $\lambda_1 = 0$). To rectify this, using a common control-theoretic technique, we employ \textit{strictification} of $P$ by adding $\epsilon > 0$ to the bottom-right entry of $P$, i.e.,

\begin{equation}
P_{\epsilon} = \begin{bmatrix}
    1 & -(1+\theta) \\
    -(1+\theta) & (1+\theta)^2+\epsilon
\end{bmatrix},
\end{equation}

\noindent
for $\epsilon > 0$. Note $\text{det}(P_{\epsilon}) = \epsilon > 0$ and $\text{tr}(P_{\epsilon}) = 1 + (1+\theta)^2+\epsilon > 0$. The eigenvalues of $P_{\epsilon}$ are $\lambda_{1,2} = \frac{[(1+\theta)^2+\epsilon+1]\pm \sqrt{[(1+\theta)^2+\epsilon+1]^2-4\epsilon}}{2}$. $\lambda_{1,2} > 0$ for all $\epsilon > 0$, so with strictification we have $P_{\epsilon} \succ 0$ as desired. We now find our ``new" Lyapunov function $\mathcal{V}_{\epsilon}(x_t) = x_{t}^TP_{\epsilon}x_t$. With $P_{\epsilon} = \begin{bmatrix}
    1 & -(1+\theta) \\
    -(1+\theta) & (1+\theta)^2+\epsilon
\end{bmatrix}$, we obtain

\begin{equation}
\mathcal{V}_{\epsilon}(x_t) = ||\Delta w_t||^2 + [(1+\theta)^2+\epsilon]||\Delta v_t||^2 - 2(1+\theta)\langle \Delta w_t, \Delta v_t \rangle.
\end{equation}

\noindent
Completing the square yields

\begin{equation}
\mathcal{V}_{\epsilon}(x_t) = ||\Delta w_t - (1+\theta)\Delta v_t||^2 + \epsilon||\Delta v_t||^2.   
\end{equation}
\end{proof}

\noindent
\textbf{Remark} In convergence analyses, the standard $1-\rho$ contraction rate for smooth quadratic NAG is $1 - \frac{1}{\sqrt{\kappa}}$, i.e., with $\rho = \frac{1}{\sqrt{\kappa}}$ (cf. \cite{LRP16}). We show that our $\mathcal{V}_{\epsilon}(x_t)$ as constructed recovers this standard NAG contraction rate.

\begin{theorem}
Let $\mathcal{V}_{\epsilon}(x_t) = ||\Delta w_t - (1+\theta)\Delta v_t||^2 + \epsilon||\Delta v_t||^2$ where $\mathcal{V}_{\epsilon}(x_{t+1}) \leq (1-\rho)\mathcal{V}_{\epsilon}(x_t)$ for $\rho > 0$, $\epsilon > 0$. Then $\rho = O(\frac{1}{\sqrt{\kappa}})$ and $1 - \rho = 1 - O(\frac{1}{\sqrt{\kappa}})$.
\end{theorem}

\begin{proof}
Let $y_t := \Delta w_t - (1+\theta)\Delta v_t$ so $\mathcal{V}_{\epsilon}(x_t) = ||y_t||^2 + \epsilon||\Delta v_t||^2$. We next find $y_{t+1} = \Delta w_{t+1} - (1+\theta)\Delta v_{t+1}$. Recall from our state dynamics that

\begin{equation}
\Delta v_{t+1} = \Delta w_t - \frac{1}{\beta}H_t\Delta w_t,
\end{equation}

\begin{equation}
\Delta w_{t+1} = (1+\theta)\Delta v_{t+1}-\theta\Delta v_t.
\end{equation}

\noindent
Then substitution yields

\begin{equation}
y_{t+1} = -\theta\Delta v_t.
\end{equation}

\noindent
Our strategy is to construct a lifted linear two-dimensional dynamical system in $(y_t, \Delta v_t) \in \mathbb{R}^2$ state space to derive $\rho$. To do so, note that $y_t$ is a function of $\Delta w_t$ so we need to decouple $y_t$ from $\Delta w_t$. Note $\Delta v_{t+1} = (I - \frac{1}{\beta}H_t)\Delta w_t$. Let $\alpha_t := I - \frac{1}{\beta}$. Then $\Delta v_{t+1} = \alpha_t \Delta w_t$ and we obtain

\begin{equation}
\Delta v_{t+1} = \alpha_t y_t + (1+\theta)\alpha_t \Delta v_t.
\end{equation}

\noindent
Let $z_{t+1} := \begin{bmatrix}
    y_{t+1} \\
    \Delta v_{t+1}
\end{bmatrix}$. Then $z_{t+1} = \begin{bmatrix}
    -\theta \Delta v_t \\
    \alpha_t y_t + (1+\theta)\alpha_t \Delta v_t
\end{bmatrix} = \begin{bmatrix}
    0 & -\theta \\
    \alpha_t & (1+\theta)\alpha_t
\end{bmatrix}\begin{bmatrix}
    y_t \\
    \Delta v_t
\end{bmatrix}$. Let $\Gamma_t := \begin{bmatrix}
    0 & -\theta \\
    \alpha_t & (1+\theta)\alpha_t
\end{bmatrix}$ so the iteration evolution of $z_t$ in this lifted state space is governed by $z_{t+1} = \Gamma_t z_t$. The characteristic equation of $\Gamma_t$ is $\lambda^2-(1+\theta)\alpha_t \lambda + \theta \alpha_t = 0$. For sufficiently large $\kappa$, we have $\theta \approx 1 - \frac{2}{\sqrt{\kappa}}$, and we fix $\alpha_t = 1 - \frac{1}{\kappa}$ corresponding to the worst-case eigen-direction. Then $\Gamma_t \approx \begin{bmatrix}
    0 & -1 + \frac{2}{\sqrt{\kappa}} \\
    1-\frac{1}{\kappa} & (2-\frac{2}{\sqrt{\kappa}})(1-\frac{1}{\kappa})
\end{bmatrix}$ where $\text{tr}(\Gamma_t) = 2 - \frac{2}{\sqrt{\kappa}}$ and $\text{det}(\Gamma_t) = 1-\frac{2}{\sqrt{\kappa}}$. Then the characteristic equation becomes

\begin{equation}
\lambda^2 - (2-\frac{2}{\sqrt{k}})\lambda + (1-\frac{2}{\sqrt{\kappa}})=0.
\end{equation}

\noindent
Solving for $\lambda_{1,2}$ yields $\lambda_{1,2} = \{1, 1- \frac{2}{\sqrt{\kappa}}\}$. Then, for $\lambda_2 = 1 - \frac{2}{\sqrt{\kappa}}$, $\rho = \frac{2}{\sqrt{\kappa}} = O(\frac{1}{\sqrt{\kappa}})$.
\end{proof}

\noindent
For uniform stability results, we want to bound $||\Delta w_t||^2$ in terms of $\mathcal{V}_{\epsilon}(x_t)$.

\begin{theorem}[Lyapunov Bound on $||\Delta w_t||^2$]
Let $\mathcal{V}_{\epsilon}(x_t) = ||\Delta w_t - (1+\theta)\Delta v_t||^2 + \epsilon||\Delta v_t||^2$ for $\epsilon > 0$. Then

\begin{center}
$||\Delta w_t||^2 \leq C_{\epsilon}\mathcal{V}_{\epsilon}(x_t)$
\end{center}

\noindent
where $C_{\epsilon} = (1 + \frac{(1+\theta)^2}{\epsilon})$.
\end{theorem}

\begin{proof}
Note $\Delta w_t = (\Delta w_t - (1+\theta)\Delta v_t)+(1+\theta)\Delta v_t$. Let $a:= \Delta w_t - (1+\theta)\Delta v_t$ and $b := (1+\theta)\Delta v_t$. Then $||\Delta w_t|| = ||a+b||$. We now make use of the following lemma:

\begin{lemma}[Young's Inequality, Inner Product Form]
Let $a$ and $b$ be vectors. Then, for all $\zeta > 0$,

\begin{center}
$||a+b||^2 \leq (1+\zeta)||a||^2 + (1 + \frac{1}{\zeta})||b||^2$.
\end{center}
\end{lemma}

\noindent
Therefore, by Young's Inequality,

\begin{equation}
||\Delta w_t||^2 \leq (1+\zeta)||\Delta w_t - (1+\theta)\Delta v_t||^2 + (1+\frac{1}{\zeta})||(1+\theta)\Delta v_t||^2.    
\end{equation}

\noindent
This is equivalent to

\begin{equation}
||\Delta w_t||^2 \leq (1+\zeta)||\Delta w_t - (1+\theta)\Delta v_t||^2 + (1+\frac{1}{\zeta})(1+\theta)^2||\Delta v_t||^2.   
\end{equation}

\noindent
Recall $\mathcal{V}_{\epsilon}(x_t) = ||\Delta w_t - (1+\theta)\Delta v_t||^2 + \epsilon||\Delta v_t||^2$. Then $\epsilon(1+\zeta) = (1+\frac{1}{\zeta})(1+\theta)^2$, which yields $\zeta = \frac{(1+\theta^2)}{\epsilon}$. This gives us

\begin{equation}
||\Delta w_t||^2 \leq (1 + \frac{(1+\theta)^2}{\epsilon})\mathcal{V}_{\epsilon}(x_t).
\end{equation}

\noindent
Let $C_{\epsilon} = (1 + \frac{(1+\theta)^2}{\epsilon})$. Then we have

\begin{equation}
||\Delta w_t||^2 \leq C_{\epsilon}\mathcal{V}_{\epsilon}(x_t).
\end{equation}
\end{proof}

\noindent
This also gives the bound $||\Delta w_t|| \leq \sqrt{C_{\epsilon}}\sqrt{\mathcal{V}_{\epsilon}(x_t)}$.

\section{Uniform Stability of Smooth Quadratic NAG via Constructed Lyapunov}

We are now in position to give a uniform stability bound for smooth quadratic NAG. Recall that \cite{CJY18} derived an $O(\frac{4\beta^2}{\gamma n}[1-(1 - \frac{1}{\sqrt{\kappa}})^T])$ uniform stability bound for smooth quadratic NAG. Via Lyapunov analysis, we derive our own uniform stability bound for smooth quadratic NAG as follows:

\begin{theorem}[Uniform Stability of Smooth Quadratic NAG]
Assume that the loss $f(w_t, z_t)$ is $\gamma$-strongly convex with $\beta$-Lipschitz/$\beta$-smooth and bounded gradients for all $z_t$. Run NAG with constant step size $\eta \leq \frac{1}{\beta}$ with fixed momentum parameter $\mu \in [0,1)$ for $T$ total iterations. Then NAG in the smooth quadratic regime is uniformly stable with

\begin{center}
$\epsilon \leq O(\frac{4G\kappa^{\frac{1}{4}}}{\beta \sqrt{n}}[\sqrt{1-(1-\rho)^T}])$
\end{center}

\noindent
for condition number $\kappa = \frac{\beta}{\gamma}$ and $\rho > 0$.
\end{theorem}

\begin{proof}
Let $S$ and $S'$ be two samples of size $n$ that differ only at a single point. We want to show that NAG is $\epsilon$-uniformly stable, i.e., by definition,

\begin{center}
$\text{sup}_z \mathbb{E}[\ell(w_t, z_{i_t}) - \ell(w'_t, z'_{i_t})] \leq \epsilon$
\end{center}

\noindent
for arbitrary point $z_{i_t}$ and for $\epsilon > 0$. Let NAG run over coupled indices over $S$ and $S'$, where $S = (z_1, z_2, ... , z_n)$ and $S' = (z_1', z_2', ... , z_n')$. We denote the index of the differing data point as $j \in \{1,2, ... , n\}$, i.e., $z_j \neq z_j'$ (so $z_i = z_i'$ for all $i \neq j$). Recall the update rules for $\beta$-smooth, quadratic NAG \cite{Bub15}:

\begin{equation}
v_{t+1} = w_t - \frac{1}{\beta} \nabla f(w_t, z_t),
\end{equation}

\begin{equation}
w_{t+1} = (1 + \frac{\sqrt{\kappa}-1}{\sqrt{\kappa}+1})v_{t+1} - (\frac{\sqrt{\kappa}-1}{\sqrt{\kappa}+1})v_t.
\end{equation}

\noindent
Setting $\theta = \frac{\sqrt{\kappa}-1}{\sqrt{\kappa}+1}$ yields

\begin{equation}
v_{t+1} = w_t - \frac{1}{\beta} \nabla \ell (w_t, z_t),
\end{equation}

\begin{equation}
w_{t+1} = (1+\theta)v_{t+1} - \theta v_t.
\end{equation}

\noindent
Note that there are two possible cases, (1) where the differing point is not selected, i.e., where $i_t \neq j$, and (2) where the differing point is selected, i.e., where $i_t = j$. We prove the uniform stability of NAG via the quadratic Lyapunov function we constructed: 

\begin{equation}
\mathcal{V}_{\epsilon}(x_t) = ||\Delta w_t - (1+\theta)\Delta v_t||^2+\epsilon||\Delta v_t||^2.
\end{equation}

\noindent
Our proof strategy is to find the expected value of the Lyapunov function, $\mathbb{E}[\mathcal{V}_{\epsilon}(x_t)]$, over both cases. We do so via the law of total expectation from probability theory, i.e., for discrete random variable $X$ with probability $P(A)$ for event $A$, where $A^c$ is the complement of $A$ with probability $P(A^c)$,

\begin{equation}
\mathbb{E}[X] = P(A)\mathbb{E}[X|A]+P(A^c)\mathbb{E}[X|A^c].
\end{equation}

\noindent
Let $A$ denote the event where the differing point is not selected, which we denote $i_t \neq j$, so $A^c$ denotes the event where the differing point is selected (i.e., $i_t = j$). We assume that the samples are uniformly distributed, which implies $P(A) = 1 - \frac{1}{n}$ and $P(A^c) = \frac{1}{n}$. Then the expected value of our Lyapunov function over both cases is

\begin{equation}
\mathbb{E}[\mathcal{V}_{\epsilon}(x_t)] = (1-\frac{1}{n})\mathbb{E}[\mathcal{V}_{\epsilon}(x_t)| i_t \neq j] + (\frac{1}{n})\mathbb{E}[\mathcal{V}_{\epsilon}(x_t)| i_t = j].
\end{equation}

\noindent
Thus, we prove by cases: \\
\\
\fbox{$i_t \neq j$} This is the case where the differing point is \textit{not} selected. Let $\{z_1, ..., z_n\}$ be uniformly distributed so $P[i_t \neq j] = 1 - \frac{1}{n}$. Note that $\Delta v_{t+1} = \Delta w_t - \frac{1}{\beta}\Delta g_t$. This implies $\Delta w_{t+1} = (1+\theta)(\Delta v_{t+1}) - \theta(\Delta v_t) = (1+\theta)(\Delta w_t - \frac{1}{\beta}\Delta g_t) - \theta(\Delta v_t) = (1+\theta)\Delta w_t - \theta(\Delta v_t) - (\frac{1+\theta}{\beta})\Delta g_t$. Recall that we derived $\Delta g_t = H_t \Delta w_t$. Substitution and factoring yields

\begin{equation}
x_{t+1} = A(H_t)x_t,
\end{equation}

\noindent
where $A(H_t) = \begin{bmatrix}
    (1+\theta)(I - \frac{H_t}{\beta}) & -\theta I \\
    I - \frac{H_t}{\beta} & 0
\end{bmatrix}$. Recall that our quadratic Lyapunov function takes the form $\mathcal{V}_{\epsilon}(x_{t+1}) = x_t^T P_{\epsilon} x_{t}$ where $P_{\epsilon} = \begin{bmatrix}
    1 & -(1+\theta) \\
    -(1+\theta) & (1+\theta)^2+\epsilon
\end{bmatrix} = P_{\epsilon}^T \succ 0$. Our Lyapunov function satisfies $\mathcal{V}_{\epsilon}(x_{t+1}) \leq (1-\rho) x_t^T P_{\epsilon} x_t$ so $\mathcal{V}_{\epsilon}(x_{t+1}) \leq (1-\rho) \mathcal{V}_{\epsilon}(x_t)$. Therefore, taking expectation in the event where $i_t \neq j$ yields

\begin{equation}
\mathbb{E}[\mathcal{V}_{\epsilon}(x_{t+1}) | i_t \neq j] \leq (1-\rho) \mathcal{V}_{\epsilon}(x_t).
\end{equation}

\noindent
\fbox{$i_t = j$} This is the case where the differing point \textit{is} selected. We keep the same uniform distribution assumption on $\{z_1, ..., z_n\}$ so $P[i_t = j] = \frac{1}{n}$. Recall $\Delta g_t = \nabla f(w_t, z_{i_t}) - \nabla f(w_t', z_{i_t})$. But here $i_t = j$ so we have $\Delta g_t = \nabla f(w_t, z_j) - \nabla f(w_t', z_j')$. Note that $\nabla f(w_t, z_j) - \nabla f(w_t', z_j') = \nabla f(w_t, z_j) - \nabla f(w_t', z_j) + \nabla f(w_t', z_j) - \nabla f(w_t', z_j')$. Thus, we have $\Delta g_t = H_t\Delta w_t + \Xi_t$ for $\Xi_t := \nabla f(w_t', z_j) - \nabla f(w_t', z_j')$. \\
\\
We can bound $||\Xi_t||$ via boundedness of the gradient. Note $||\Xi_t|| = ||\nabla f(w_t', z_j) - \nabla f(w_t', z_j')||$. Then by the Triangle Inequality $||\Xi_t|| \leq ||\nabla f(w_t', z_j)|| + ||\nabla f(w_t', z_j')||$. But by boundedness of the gradient, $||\nabla f(w_t', z_j)|| \leq G$ and $||\nabla f (w_t', z_j')|| \leq G$. Thus, $||\Xi_t|| \leq 2G$. \\
\\
Recall that $\Delta v_{t+1} = \Delta w_t - \frac{1}{\beta}\Delta g_t$. Then substituting $\Delta g_t = H_t \Delta w_t + \Xi_t$ yields $\Delta v_{t+1} = (I - \frac{H_t}{\beta})\Delta w_t - \frac{1}{\beta} \Xi_t$. Moreover, for $\Delta w_{t+1} = (1+\theta)\Delta v_{t+1} - \theta \Delta v_t$, this yields $\Delta w_{t+1} = (1+\theta)(I - \frac{H_t}{\beta})\Delta w_t - \theta \Delta v_t - (\frac{1+\theta}{\beta})\Xi_t$. Thus, the iteration evolution of the state vector $x_t$ takes the form

\begin{equation}
x_{t+1} = A(H_t)x_t + B\Xi_t,
\end{equation}

\noindent
where $A(H_t) = \begin{bmatrix}
    (1+\theta)(I-\frac{H_t}{\beta}) & -\theta I \\
    I-\frac{H_t}{\beta} & 0
\end{bmatrix}$ and $B = \begin{bmatrix}
    -\frac{(1+\theta)}{\beta} \\
    -\frac{1}{\beta}
\end{bmatrix}$. By change of metric from standard Euclidean $L^2$ norm to the $P_{\epsilon}$ norm, our Lyapunov function takes the form $\mathcal{V}_{\epsilon}(x_{t+1}) = ||x_{t+1}||^2_{P_{\epsilon}}$. Then $\mathcal{V}_{\epsilon}(x_{t+1}) = ||A(H_t)x_t + B\Xi_t||^2_{P_{\epsilon}}$ and by Young's Inequality we have

\begin{equation}
||A(H_t)x_t + B\Xi_t||^2_{P_{\epsilon}} \leq (1 + \zeta)||A(H_t)x_t||^2_{P_{\epsilon}} + (1 + \frac{1}{\zeta})||B\Xi_t||^2_{P_{\epsilon}}
\end{equation}

\noindent
for $\zeta > 0$. Note that $||\Xi_t|| \leq 2G$ only in the $L^2$ norm, so this bound is not guaranteed in the new $P_{\epsilon}$ norm. We want to bound $||B\Xi_t||^2_{P_{\epsilon}}$. By definition, $||B\Xi_t||^2_{P_{\epsilon}} = (B\Xi_t)^T P_{\epsilon} (B\Xi_t)$. Computation yields $||B\Xi_t||^2_{P_{\epsilon}} = (\frac{||\Xi_t||^2}{\beta^2})\epsilon \leq (\frac{(2G)^2}{\beta^2})(\epsilon) = \frac{4G^2\epsilon}{\beta^2}$. We next bound $||A(H_t)x_t||^2_{P_{\epsilon}}$. Note $||A(H_t)x_t||^2_{P_{\epsilon}} = x_t^T(A(H_t)^TP_{\epsilon}A(H_t))x_t$. But by construction there exists $\epsilon >0$, $\rho >0$ such that $A(H_t)^TP_{\epsilon}A(H_t) \preceq (1-\rho)P_{\epsilon}$. Then it follows that $||A(H_t)x_t||^2_{P_{\epsilon}} \leq (1-\rho)\mathcal{V}_{\epsilon}(x_t)$. Thus, in the case where $i_t = j$, we have

\begin{equation}
\mathbb{E}[\mathcal{V}_{\epsilon}(x_{t+1})|i_t = j] \leq (1+\zeta)(1-\rho)\mathcal{V}_{\epsilon}(x_t) + (1 + \frac{1}{\zeta})(\frac{4G^2\epsilon}{\beta^2}).
\end{equation}

\noindent
Then by the Law of Total Expectation we have

\begin{equation}
\mathbb{E}[\mathcal{V}_{\epsilon}(x_{t+1})] \leq (1-\frac{1}{n})(1-\rho)\mathcal{V}_{\epsilon}(x_t) + \frac{1}{n}[(1+\zeta)(1-\rho)\mathcal{V}_{\epsilon}(x_t) + (1 + \frac{1}{\zeta})(\frac{4G^2\epsilon}{\beta^2})],
\end{equation}

\noindent
for $\zeta, \rho, \epsilon >0$. Algebraic manipulation yields

\begin{equation}
\mathbb{E}[\mathcal{V}_{\epsilon}(x_{t+1})] \leq (1-\rho)(1 + \frac{\zeta}{n})\mathcal{V}_{\epsilon}(x_t) + \frac{(1+\frac{1}{\zeta})(4G^2\epsilon)}{n\beta^2}.
\end{equation}

\noindent
As $n \rightarrow \infty$, $(1-\rho)(1+\frac{\zeta}{n}) \rightarrow (1-\rho)$. Let $\delta_t := \mathcal{V}_{\epsilon}(x_t)$ and let $C := \frac{(1+\frac{1}{\zeta})(4G^2\epsilon)}{n\beta^2}$. Then we have a linear recurrence of the form

\begin{equation}
\delta_{t+1} \leq (1-\rho)\delta_t + C
\end{equation}

\noindent
for sufficiently large $n$. Since each of $S$ and $S'$ are initialized identically, we have $\delta_0 = 0$ so this linear recurrence is unrolled as

\begin{equation}
\delta_T \leq C\sum_{k=0}^{T-1}(1-\rho)^k.
\end{equation}

\noindent
$|1-\rho| < 1$ so this geometric series converges. Thus, we have

\begin{equation}
\delta_T \leq C(\frac{1-(1-\rho)^T}{\rho}).
\end{equation}

\noindent
This yields $\mathcal{V}_{\epsilon}(x_t) \leq (\frac{(1+\frac{1}{\zeta})(4G^2\epsilon)}{n\beta^2})(\frac{1-(1-\rho)^T}{\rho})$. But we derived $||\Delta w_T||^2 \leq C_{\epsilon}\mathcal{V}_{\epsilon}(x_t)$ where $C_{\epsilon} = 1 + \frac{(1+\theta)^2}{\epsilon}$. Then we have

\begin{equation}
\delta_T^2 \leq (1 + \frac{(1+\theta)^2}{\epsilon})(\frac{(1+\frac{1}{\zeta})(4G^2\epsilon)}{n\beta^2})(\frac{1-(1-\rho)^T}{\rho}).
\end{equation}

\noindent
Factoring and absorbing constants yields

\begin{equation}
\delta_T^2 \leq (\frac{16G^2\sqrt{\kappa}}{n \beta^2})[1-(1-\rho)^T].
\end{equation}

\noindent
Then by $G$-Lipschitzness of the loss function and taking the square root we have

\begin{equation}
\delta_T \leq \frac{4G\kappa^{\frac{1}{4}}}{\beta \sqrt{n}}[\sqrt{1-(1-\rho)^T}].
\end{equation}

\noindent
Therefore, smooth quadratic NAG is uniformly stable with $\epsilon \leq O(\frac{1}{\sqrt{n}})$.
\end{proof}

\noindent
\textbf{Remark} Our uniform stability bound for smooth quadratic NAG is $O(\frac{1}{\sqrt{n}})$, whereas the one derived by \cite{CJY18} is $O(\frac{1}{n})$. Moreover, as $T \rightarrow \infty$, we obtain a uniform stability bound $O(\frac{4Gk^{\frac{1}{4}}}{\beta \sqrt{n}})$ (i.e., with no consideration of ``early stopping" as done in practice). Taking the same limit for the bound derived by \cite{CJY18} yields an $O(\frac{4\beta^2}{\gamma n})$ bound.

\section{Numerical Experiments}

We attempt to empirically validate our theoretical result with small-scale numerical experiments. We derived a $O(\frac{4Gk^{\frac{1}{4}}}{\beta \sqrt{n}}[\sqrt{1-(1-\rho)^T}])$ uniform stability bound for smooth quadratic NAG. As does \cite{CJY18}, we do so via a simple logistic regression on the Wisconsin breast cancer dataset (via Scikit-learn) with $n=569$ and 30 feature variables. We conduct two sets of experiments:

\begin{enumerate}
    \item \textbf{(Empirical) Algorithmic Stability vs. Number of Samples} First, we subsample the dataset at different sizes, plotting empirical stability as a function of number of samples $n$. Since we derived an $O(\frac{4Gk^{\frac{1}{4}}}{\beta \sqrt{n}})$ bound (assuming a sufficiently large number of iterations), we should expect empirical stability to scale roughly $\frac{1}{\sqrt{n}}$, with a log-log plot of empirical stability as a function of number of samples to be linear with slope $\approx - \frac{1}{2}$.

    \item \textbf{(Empirical) Algorithmic Stability vs. Number of Iterations} Then, we verify iteration independence by plotting empirical stability as a function of number of iterations. Recall we obtained an $O(\frac{4G\kappa^{\frac{1}{4}}}{\beta \sqrt{n}}[\sqrt{1-(1-\rho)^T}])$ uniform stability bound. Suppose $g$ is a function of algorithmic stability. Then given our derived bound we should expect here that $g(T) \sim\sqrt{T}$ for total number of iterations $T$.
\end{enumerate}

\noindent
For both experiments, we use the sigmoid activation function and the logistic loss function. We run NAG with a learning rate $\eta = 0.01$ and momentum parameter $\mu = 0.9$ over $T = 10,000$ iterations. To attempt to enforce strong convexity on the loss function, we add a regularization term to the loss function with regularization parameter $\lambda = 1 \times 10^{-3}$. We repeat each of these experiments for 25 independent trials. As a metric of algorithmic (uniform) stability, we compute \textit{empirical algorithmic stability}, i.e., the absolute value of the difference of parameter iterate values $w$ and $w'$ from $S$ and $S'$, respectively.

\subsection{Algorithmic Stability vs. Number of Samples}

For this experiment, we use the same subset sizes of 50, 100, 200, and 400, and run NAG again on $S$ and $S'$ constructed from the Wisconsin breast cancer dataset over 25 independent trials. In this experiment, we test whether empirical algorithmic stability is iteration-independent, yielding the following results:

\begin{figure}[H]
    \centering
    \includegraphics[width=1\linewidth]{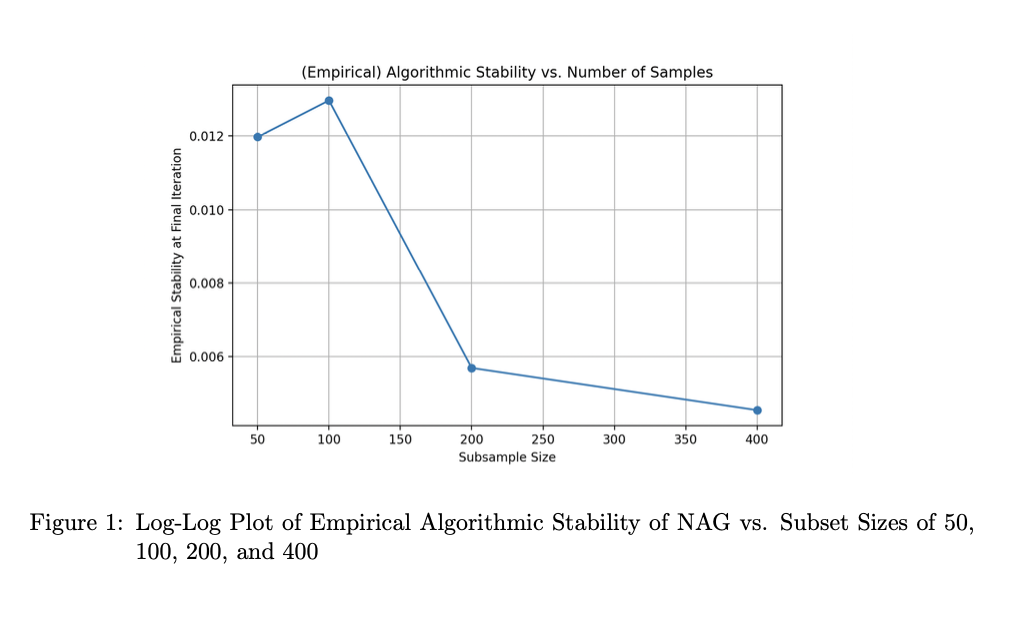}
\end{figure}

\noindent
The trend line here has slope $\approx -\frac{1}{2}$, consistent with our derived $O(\frac{4Gk^{\frac{1}{4}}}{\beta \sqrt{n}})$ bound for sufficiently large $T$.

\subsection{Algorithmic Stability vs. Number of Iterations}

For this experiment, we use the same subset sizes of 50, 100, 200, and 400, and run NAG again on $S$ and $S'$ constructed from the Wisconsin breast cancer dataset over 25 independent trials. In this experiment, we test whether empirical algorithmic stability is iteration-independent, yielding the following results:

\begin{figure}[H]
    \centering
    \includegraphics[width=1\linewidth]{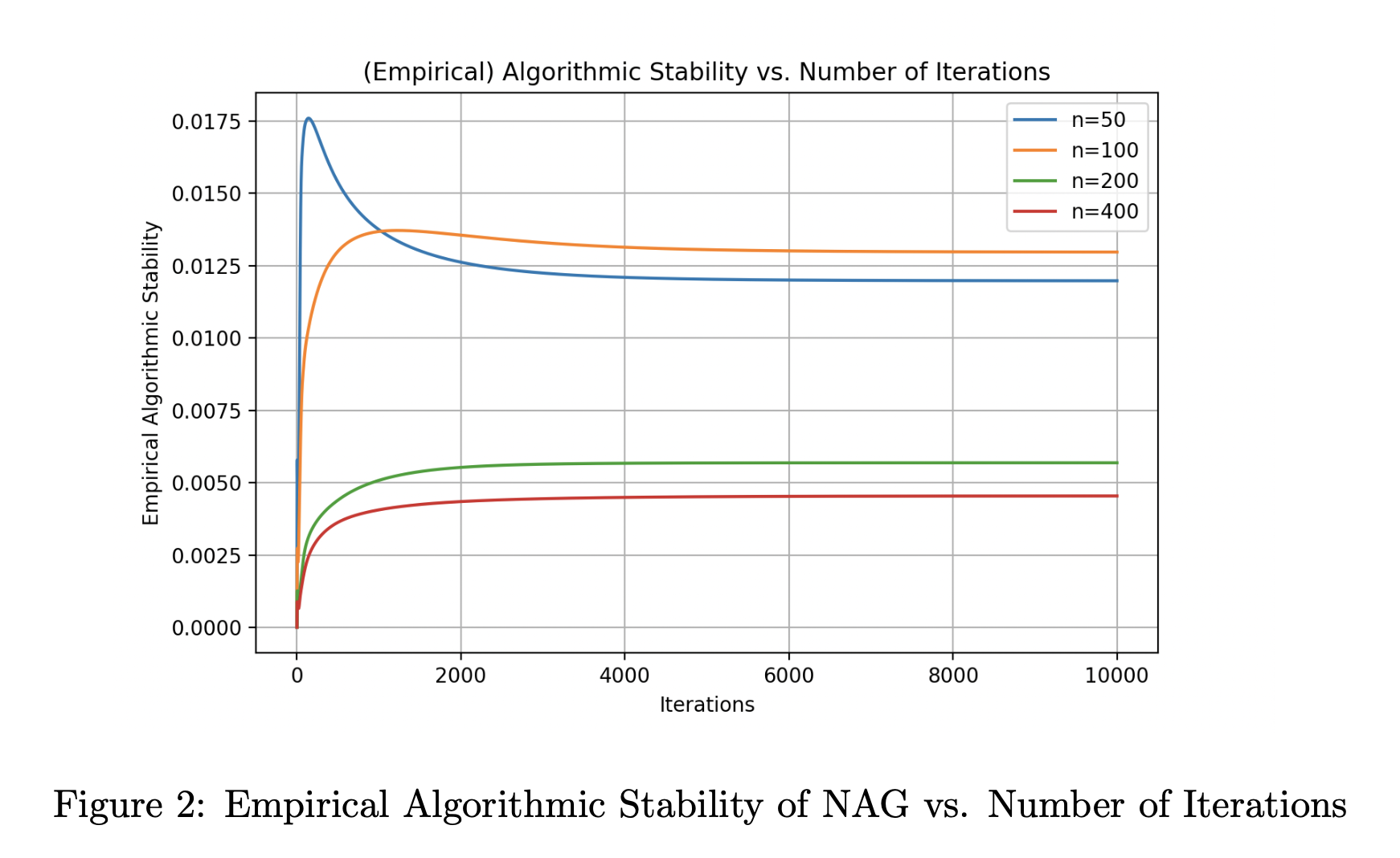}
\end{figure}

\noindent
For each of the four subset sizes, we obtain roughly $g(T) \sim \sqrt{T}$ for $T$ number of iterations, which coheres with our derived $O(\frac{4Gk^{\frac{1}{4}}}{\beta \sqrt{n}}[\sqrt{1-(1-\rho)^T}])$ uniform stability bound.

\section{Semi-Definite Programming for Automated Stability Certificates}

Our method of deriving the Lyapunov function ``by hand" to prove the uniform stability of smooth quadratic NAG was rather cumbersome. Here, we detail an approach that leverages the smooth quadratic regime to automate this method that can also be adapted to other first-order (accelerated) optimizers of similar structure. \\
\\
\cite{LRP16} show how first-order optimizers can be represented as \textit{Lur'e} systems, i.e., a linear dynamical system of the form

\begin{equation}
    \begin{cases}
        \zeta_{k+1} = A\zeta_k + Bu_k, \\
        y_k = C\zeta_k + Du_k,
    \end{cases}
\end{equation}

\noindent
for inputs $u_k \in \mathbb{R}^d$, outputs $y_k \in \mathbb{R}^d$, and state vectors $\zeta_k \in \mathbb{R}^m$. All first-order optimizers use the gradient as its oracle, and introducing the gradient (operator) consequently introduces non-linearity into the system. As \cite{LRP16} show, let $\phi(y) := \nabla f(y)$. Then a first-order optimizer with non-linearity $\phi$ takes the form

\begin{equation}
    \begin{cases}
        \zeta_{k+1} = A\zeta_k + Bu_k, \\
        y_k = C\zeta_k + Du_k, \\
        u_k = \phi(y_k).
    \end{cases}
\end{equation}

\noindent
NAG can be formulated as a Lur'e system in this framework from \cite{LRP16}. Let $x_t := \begin{bmatrix}
    w_t \\
    w_{t-1}
\end{bmatrix}$. Then smooth quadratic NAG in Lur'e form is

\begin{equation}
    \begin{cases}
        x_{t+1} = Ax_t + Bu_t, \\
        y_t = Cx_t + Du_t, \\
        u_t = \phi(y_t),
    \end{cases}
\end{equation}

\noindent
with $A = \begin{bmatrix}
    (1+\theta)I & -\theta I \\
    I & 0
\end{bmatrix}$, $B = \begin{bmatrix}
    -\eta I \\
    0
\end{bmatrix}$, $C = \begin{bmatrix}
    (1+\theta)I & -\theta I
\end{bmatrix}$, and $D = 0$. \\
\\
Note that, in the smooth quadratic regime, $\gamma$-strong convexity and $\beta$-smoothness of the loss $\ell(w_t,z_t)$ both hold. Since smooth quadratic NAG in Lur'e form is representable in an affine space, finding $P \succ 0$ is achievable via semi-definite programming (SDP) (cf. Lecture 4 of \cite{BN01} for a comprehensive overview of SDP). Recall $x_t := \begin{bmatrix}
    \Delta w_t \\
    \Delta v_t
\end{bmatrix} = \begin{bmatrix}
    w_t - w_t' \\
    v_t - v_t'
\end{bmatrix}$ and that we prove the uniform stability of smooth quadratic NAG via the quadratic Lyapunov function $\mathcal{V}(x_t) = x_t^T Px_t$. We construct an SDP that finds $P \succ 0$, $\lambda > 0$ such that $\mathcal{V}(x_t) = x_t^T Px_t$ satisfies

\begin{equation}
\mathcal{V}(x_{t+1}) - \mathcal{V}(x_t) \leq -\lambda||x_t||^2
\end{equation}

\noindent
for any $S$, $S'$. Let $\Delta V:= \mathcal{V}(x_{t+1}) - \mathcal{V}(x_t)$. Then our SDP finds $P \succ 0$, $\lambda > 0$ such that

\begin{equation}
\Delta V \leq -\lambda ||x_t||^2.
\end{equation}

\noindent
Recall that, in Lur'e form, $x_{t+1} = Ax_t + Bu_t$, which implies $\mathcal{V}(x_{t+1}) = (Ax_t + Bu_t)^T P (Ax_t + Bu_t) = x_t^TA^TPAx_t + 2x_t^TA^TPBu_t + u_t^TB^TPBu_t$. This implies

\begin{equation}
\Delta V = x_t^T(A^TPA-P)x_t + 2x_t^TA^TPBu_t + u_t^TB^TPBu_t.
\end{equation}

\noindent
Let $z_t := \begin{bmatrix}
    x_t \\
    u_t
\end{bmatrix}$. Then we have

\begin{equation}
\Delta V = z_t^T\begin{bmatrix}
    A^TPA-P & A^TPB \\
    B^TPA & B^TPB
\end{bmatrix}z_t
\end{equation}

\noindent
with $M := \begin{bmatrix}
    A^TPA-P & A^TPB \\
    B^TPA & B^TPB
\end{bmatrix}$. Recall that the condition for our SDP is $\Delta V \leq -\lambda ||x_t||^2$ for $\lambda > 0$. Note that $-\lambda ||x_t||^2 = -z_t^T \begin{bmatrix}
    \lambda I & 0 \\
    0 & 0
\end{bmatrix}z_t$. Thus, we have

\begin{equation}
z_t^T\begin{bmatrix}
    A^TPA-P+\lambda I & A^TPB \\
    B^TPA & B^TPB
\end{bmatrix}z_t \leq 0. 
\end{equation}

\noindent
Let $M' = \begin{bmatrix}
    A^TPA-P+\lambda I & A^TPB \\
    B^TPA & B^TPB
\end{bmatrix}$. Then this condition takes the form

\begin{equation}
z_t^T M' z_t \leq 0.
\end{equation}

\noindent
Recall $z_t := \begin{bmatrix}
    x_t \\
    u_t
\end{bmatrix}$ with $u_t = \phi(y_t) = \nabla f(y)$ by our Lur'e construction. Since we are in the smooth quadratic regime, we know $\gamma I \preceq \nabla^2f \preceq \beta I$, which thus constrains the set of all possible $z_t$ via the following \textit{sector inequality/IQC}:

\begin{theorem}[Smooth Quadratic Gradient Sector IQC]
Let $f$ be both $\gamma$-strongly convex and $\beta$-smooth, and let $\phi(y_t) = \nabla f (y_t,z_t)$. Let $\Delta y_t := y_t-y_t'$, $\Delta u_t := \phi(y_t)-\phi(y_t')$, and $\xi_t := \begin{bmatrix}
    \Delta y_t \\
    \Delta u_t
\end{bmatrix}$. Then any admissible trajectory of $(\Delta y_t, \Delta u_t)$ must satisfy

\begin{center}
$\xi_t^T(\tau_1\Pi_1 + \tau_2\Pi_2)\xi_t \geq 0$
\end{center}

\noindent
for $\Pi_1 = \begin{bmatrix}
    -\gamma I & \frac{1}{2}I \\
    \frac{1}{2}I & 0
\end{bmatrix}$, $\Pi_2 = \begin{bmatrix}
    0 & \frac{1}{2}I \\
    \frac{1}{2}I & -\frac{1}{\beta}I
\end{bmatrix}$, and $\tau_1, \tau_2 \geq 0$.
\end{theorem}

\begin{proof}
$f$ is both $\gamma$-strongly convex and $\beta$-smooth, so we make use of the following result from \cite{BH77}:

\begin{lemma}[Baillon-Haddad Theorem]
Let $f$ be convex, continuously differentiable, and defined everywhere on a Hilbert space with $\beta$-Lipschitz gradients. Then $\forall x,y \in \text{dom}(f)$,

\begin{center}
$\langle \nabla f(x) - \nabla f(y), x-y \rangle \geq \frac{1}{\beta}||\nabla f(x) - \nabla f(y)||^2$.
\end{center}
\end{lemma}

\noindent
The Baillon-Haddad Theorem is also known as \textit{co-coercivity of the gradient}. Since $\Delta y_t := y_t-y_t'$ and $\Delta u_t := \phi(y_t)-\phi(y_t')$, we have

\begin{equation}
\langle \Delta u_t, \Delta y_t \rangle \geq \frac{1}{\beta}||\Delta u_t||^2.
\end{equation}

\noindent
Moreover, since $f$ is $\gamma$-strongly convex, by corollary $f$ also satisfies monotonicity of the gradient, i.e., $\forall x,y \in \text{dom}(f)$, $\langle \nabla f(x) - \nabla f(y), x-y\rangle \geq \gamma||x-y||^2$. Consequently, we also have

\begin{equation}
\langle \Delta u_t, \Delta y_t\rangle \geq \gamma||\Delta y_t||^2.
\end{equation}

\noindent
This implies 

\begin{equation}
\langle \Delta u_t, \Delta y_t \rangle - \gamma ||\Delta y_t||^2 \geq 0.
\end{equation}

\noindent
Let $\xi_t := \begin{bmatrix}
    \Delta y_t \\
    \Delta u_t
\end{bmatrix}$. Then we have the IQC for the $\gamma$-strong convexity condition, i.e., 

\begin{equation}
\xi_t^T \begin{bmatrix}
    -\gamma I & \frac{1}{2}I \\
    \frac{1}{2}I & 0
\end{bmatrix}\xi_t \geq 0,
\end{equation}

\noindent
where $\Pi_1 = \begin{bmatrix}
    -\gamma I & \frac{1}{2}I \\
    \frac{1}{2}I & 0
\end{bmatrix}$. Recall $\langle \Delta u_t, \Delta y_t \rangle \geq \frac{1}{\beta} ||\Delta u_t||^2$ by the Baillon-Haddad Theorem. This implies

\begin{equation}
\langle \Delta u_t, \Delta y_t \rangle - \frac{1}{\beta}||\Delta u_t||^2 \geq 0.
\end{equation}

\noindent
Then the IQC for $\beta$-smoothness condition is

\begin{equation}
\xi_t^T\begin{bmatrix}
    0 & \frac{1}{2}I \\
    \frac{1}{2}I & -\frac{1}{\beta}I
\end{bmatrix}\xi_t \geq 0,
\end{equation}

\noindent
with $\Pi_2 = \begin{bmatrix}
    0 & \frac{1}{2}I \\
    \frac{1}{2}I & -\frac{1}{\beta}I
\end{bmatrix}$. Any admissible trajectory of $(\Delta y_t, \Delta u_t)$ must satisfy each IQC, so any admissible trajectory of $(\Delta y_t, \Delta u_t)$ likewise satisfies any non-negative combination of these IQC's, i.e., for $\tau_1, \tau_2 \geq 0$,

\begin{equation}
\xi_t^T(\tau_1\Pi_1 + \tau_2\Pi_2)\xi_t \geq 0.
\end{equation}

\noindent
Let $\Pi := \tau_1\Pi_1 + \tau_2\Pi_2$. Then any admissible trajectory of $(\Delta y_t, \Delta u_t)$ satisfies

\begin{equation}
\xi_t^T\Pi\xi_t \geq 0.
\end{equation}
\end{proof}

\noindent
In Lur'e form, $y_t = Cx_t + Du_t$ so $\Delta y_t = C\Delta x_t + D \Delta u_t$. But recall $z_t := \begin{bmatrix}
    x_t \\
    u_t
\end{bmatrix}$ and $\xi_t := \begin{bmatrix}
    \Delta y_t \\
    \Delta u_t
\end{bmatrix}$. This implies $\xi_t = \begin{bmatrix}
    C & D \\
    0 & I
\end{bmatrix}z_t$. Let $J := \begin{bmatrix}
    C & D \\
    0 & I
\end{bmatrix}$ so $\xi_t = Jz_t$. Consequently, the IQC takes the form

\begin{equation}
\xi_t^TJ^T\Pi J\xi_t \geq 0.
\end{equation}

\noindent
Our Lyapunov strict contraction condition takes the form $z_t^T M' z_t \leq 0$, i.e., as a certificate of uniform stability, we show that $z_t^T M' z_t \leq 0$ for all trajectories that satisfy $\xi_t^TJ^T\Pi J\xi_t \geq 0$, i.e., if $\xi_t^TJ^T\Pi J\xi_t \geq 0$, then $z_t^T M' z_t \leq 0$. This implication holds if the \textit{S-lemma} holds. The original result is due to \cite{Yak62}.

\begin{theorem}[S-Lemma, Lur'e/IQC Formulation]
Let $Q := J^T\Pi J$. If there exist $\tau_1, \tau_2 \geq 0$ such that

\begin{center}
$M' + Q \preceq 0$,
\end{center}

\noindent
then $z_t^T Q z_t \geq 0$ implies $z_t^T M' z_t \leq 0$ for all $z_t$.
\end{theorem}

\noindent
With the S-Lemma, we can now state the SDP:

\begin{align}
    \text{find} \quad & P \succ 0, \lambda > 0, \tau_1 \geq 0, \tau_2 \geq 0 \\
    \text{s.t.} \quad & \begin{bmatrix}
        A^TPA-P+\lambda I & A^TPB \\
        B^TPA & B^TPB
    \end{bmatrix} + J^T(\tau_1 \Pi_1 + \tau_2 \Pi_2)J \preceq 0
\end{align}

\noindent
In the case of smooth quadratic NAG, $\Pi_1 = \begin{bmatrix}
    -\gamma I & \frac{1}{2}I \\
    \frac{1}{2}I & 0
\end{bmatrix}$, $\Pi_2 = \begin{bmatrix}
    0 & \frac{1}{2}I \\
    \frac{1}{2}I & -\frac{1}{\beta}I
\end{bmatrix}$, and $J = \begin{bmatrix}
    C & D \\
    0 & I
\end{bmatrix}$.

\section{Discussion}

We have shown how Lyapunov functions can be extended from convergence to stability analysis of first-order optimization algorithms, thus permitting the case-by-case iterate uniform stability arguments made by \cite{HRS16} for stochastic gradient descent (SGD) to be extended to accelerated optimizers and any other first-order optimizer with higher-order dynamics expressible as a Lur'e system. Instantiating this framework for Nesterov Accelerated Gradient (NAG), we derived an $O(\frac{4Gk^{\frac{1}{4}}}{\beta \sqrt{n}}[\sqrt{1-(1-\rho)^T}])$ uniform stability bound for smooth quadratic NAG, roughly matching the uniform stability bound for NAG derived by \cite{CJY18} in the same regime. Moreover, we show how our Lyapunov construction for deriving uniform stability bounds for first-order (accelerated) optimizers lends itself to an IQC formulation that allows for the automation of deriving such uniform stability bounds via semi-definite programming (SDP). \\
\\
Since the Lypapunov-IQC framework detailed here lends to numerical verification of algorithmic stability, a clear research direction is work towards actual software implementation of such SDP's (e.g., with integration in to CVPXY, MOSEK, etc.). Moreover, another clear future research direction is to extend such analysis to the smooth general convex case and eventually the non-convex setting. However, as \cite{AK21} note, this is complicated by not having a fixed $H_t$ over the execution of the algorithm over each of $S$ and $S'$.

\newpage

\appendix
\section{Uniform Stability of Smooth Quadratic SGD}

\cite{HRS16} originally proved the uniform stability of stochastic gradient descent (SGD) in the smooth quadratic regime. We present the proof here for reference as follows:

\begin{theorem}
Assume that the loss $f_i(w)$ is $\gamma$-strongly convex and $\beta$-smooth for all $z$. Run SGD with constant $\eta \leq \frac{\gamma}{\beta^2}$ for $T$ iterations. Then SGD is uniformly stable with

\begin{center}
$\epsilon \leq \frac{2G^2}{\gamma n}$.
\end{center}
\end{theorem}

\begin{proof}
Let $S$ and $S'$ be two samples of size $n$ that differ only at a single point. We want to show that SGD is $\epsilon$-uniformly stable, i.e., by definition,

\begin{center}
$\text{sup}_z \mathbb{E}[\ell(A(S), z) - \ell(A(S'), z)] \leq \epsilon$
\end{center}

\noindent
for arbitrary point $z$. Let SGD run over coupled indices over $S$ and $S'$, where $S = (z_1, z_2, ... , z_n)$ and $S' = (z_1', z_2', ... , z_n')$. We denote the index of the differing data point as $j \in \{1,2, ... , n\}$, i.e., $z_j \neq z_j'$ (so $z_i = z_i'$ for all $i \neq j$). \\
\\
Recall the update rule for SGD for some arbitrary parameter $x_t$, i.e., 

\begin{center}
$x_{t+1} = x_t - \eta \nabla f_i(x_t)$,
\end{center}

\noindent
where $\eta$ is the learning rate/step size and $\nabla f_i(x_t)$ is the gradient at $x_t$. Run SGD over both $S$ and $S'$. In doing so, we obtain two separate update rules:

\begin{equation}
w_{t+1} = w_t - \eta \nabla f_i(w_t),
\end{equation}

\begin{equation}
w'_{t+1} = w'_t - \eta \nabla f_i(w'_t).
\end{equation}

\noindent
We now define an iterate divergence term,

\begin{equation}
\delta_t = \mathbb{E}[||w_t - w'_t||].
\end{equation}

\noindent
By definition, SGD is stable iff $\delta_t = O(\frac{1}{n})$. Note that there are two possible cases, (1) where the differing point is \textit{not} selected, i.e., where $i_t \neq j$, and (2) where the differing point \textit{is} selected, i.e., where $i_t = j$. Thus, we use proof by cases. \\
\\
\fbox{\{$i_t \neq j$\}} Let $\{z_1, z_2, ... , z_n\}$ be uniformly distributed, so $\mathbb{P}[i_t \neq j] = 1 - \frac{1}{n}$. Then we obtain Eqns. 3 \& 4 as the SGD update rules with the loss function and gradient values. Thus, $\delta_{t+1} = ||w_{t+1} - w'_{t+1}|| = ||(w_t - w'_t) - \eta(\nabla \ell(w_t, z_{i_t}) - \nabla \ell(w'_t, z'_{i_t}))||$. This implies $||w_{t+1} - w'_{t+1}||^2 = ||(w_t - w'_t) - \eta(\nabla \ell(w_t, z_{i_t}) - \nabla \ell(w'_t, z'_{i_t}))||^2 = || w_{t} - w'_{t}||^2 - 2\eta \langle w_t - w'_t, \nabla \ell(w_t, z_{i_t}) - \nabla \ell(w'_t, z'_{i_t} \rangle + \eta^2|| \nabla \ell(w_t, z_{i_t}) - \nabla \ell(w'_t, z'_{i_t})||^2$, i.e., 

\begin{equation}
||w_{t+1} - w'_{t+1}||^2 = || w_{t} - w'_{t}||^2 -2\eta \langle w_t - w'_t, \nabla \ell(w_t, z_{i_t}) - \nabla \ell(w'_t, z'_{i_t} \rangle + \eta^2|| \nabla \ell(w_t, z_{i_t}) - \nabla \ell(w'_t, z'_{i_t})||^2.
\end{equation}

\noindent
We bound the middle term, $-2\eta \langle w_t - w'_t, \nabla \ell(w_t, z_{i_t}) - \nabla \ell(w'_t, z'_{i_t} \rangle$, which we do via gradient monotonicity. Doing so yields

\begin{equation}
-2\eta \langle w_t - w'_t, \nabla \ell(w_t, z_{i_t}) - \nabla \ell(w'_t, z'_{i_t}) \rangle \leq -2\eta\gamma || w_t - w_t' ||^2.
\end{equation}

\noindent
We can now bound the last term, $\eta^2|| \nabla \ell(w_t, z_{i_t}) - \nabla \ell(w'_t, z'_{i_t})||^2$, which we do via $\beta$-smoothness. Directly applying the definition of $\beta$-smoothness yields

\begin{equation}
\eta^2|| \nabla \ell(w_t, z_{i_t}) - \nabla \ell(w'_t, z'_{i_t})||^2 \leq \eta^2 \beta^2|| w_t - w'_t ||^2.
\end{equation}

\noindent
These inequalities thus yield

\begin{equation}
||w_{t+1} - w'_{t+1}||^2 \leq ||w_{t} - w'_{t}||^2 - 2\eta\gamma|| w_t - w_t' ||^2 + \eta^2 \beta^2|| w_t - w'_t ||^2.
\end{equation}

\noindent
Factoring $||w_{t} - w'_{t}||^2$ on the right-hand side of the inequality yields

\begin{equation}
||w_{t+1} - w'_{t+1}||^2 \leq (1-2\eta\gamma+\eta^2\beta^2)||w_{t} - w'_{t}||^2.
\end{equation}

\noindent
Optimal $\eta$ occurs where $\frac{\delta}{\delta \eta}[1-2\eta\gamma+\eta^2\beta^2]=0$, which suggests $\eta \leq \frac{\gamma}{\beta^2}$. Then for $\eta \leq \frac{\gamma}{\beta^2}$ we have 

\begin{equation}
||w_{t+1} - w'_{t+1}||^2 \leq (1-\eta\gamma)||w_{t} - w'_{t}||^2.
\end{equation}

\noindent
Then taking the square root with a looser bound yields

\begin{equation}
||w_{t+1} - w'_{t+1}|| \leq (1-\eta\gamma)||w_{t} - w'_{t}||.
\end{equation}

\noindent
Thus, in the case where $\{i_t \neq j\}$, the iterate divergence decreases from one iteration to the next at a factor of $1 - \eta\gamma$. \\
\\
\fbox{$\{i_t = j\}$} We keep the same uniform distribution assumption for both $S$ and $S'$, where SGD is ran over a coupled set of indices over both $S$ and $S'$, so $\mathbb{P}[i_t = j] = \frac{1}{n}$. Recall $\delta_{t+1} = ||w_{t+1} - w'_{t+1}|| = ||(w_t - w'_t) - \eta(\nabla \ell(w_t, z_{i_t}) - \nabla \ell(w'_t, z'_{i_t}))||$. Applying the triangle inequality implies

\begin{equation}
||w_{t+1} - w'_{t+1}|| \leq ||(w_t - w'_t)|| +  \eta(\nabla \ell(w_t, z_{i_t}) - \nabla \ell(w'_t, z'_{i_t}))||.
\end{equation}

\noindent
We can then apply the triangle inequality again to the $(\nabla \ell(w_t, z_{i_t}) - \nabla \ell(w'_t, z'_{i_t}))||$, which yields

\begin{equation}
||w_{t+1} - w'_{t+1}|| \leq ||(w_t - w'_t)|| +  \eta||\nabla \ell(w_t, z_{i_t})|| + \eta ||\nabla \ell(w'_t, z'_{i_t}))||.  
\end{equation}

\noindent
By boundedness of the gradient, we can infer $||\nabla \ell(w_t, z_{i_t})|| \leq G$ and $||\nabla \ell(w'_t, z'_{i_t}))|| \leq G$, which implies

\begin{equation}
||w_{t+1} - w'_{t+1}|| \leq ||(w_t - w'_t)|| +  \eta G + \eta G = ||(w_t - w'_t)|| + 2\eta G.  
\end{equation}

\noindent
Let $\mathbb{E}[\delta_t] = ||(w_t - w'_t)||$. Taking expectation over (100) yields

\begin{equation}
\mathbb{E}[||w_{t+1} - w'_{t+1}|| | i_t = j] \leq \delta_t + 2 \eta G.
\end{equation}

\noindent
Moreover, we let $\delta_{t+1} = \mathbb{E}[||w_{t+1} - w'_{t+1}|| | i_t = j]$. Thus, for $i_t = j$, 

\begin{equation}
\delta_{t+1} \leq \delta_t + 2 \eta G.
\end{equation}

\noindent
Recall the law of total expectation for a discrete random variable $X$ with probability $P(A)$ for event $A$, where $A^c$ is the complement of $A$ with $P(X^c) = P(A^c)$, i.e., $\mathbb{E}[X] = P(A)\mathbb{E}[X|A] + P(A^c)\mathbb{E}[X|A^c]$. Here, $A$ is the event where $i_t \neq j$, so $A^c$ is the event where $i_t = j$. Thus, applying the law of total expectation yields

\begin{equation}
\delta_{t+1} = (1-\frac{1}{n})\mathbb{E}[||w_{t+1} - w'_{t+1}|| | i_t \neq j] + \frac{1}{n}\mathbb{E}[||w_{t+1} - w'_{t+1}|| | i_t = j].
\end{equation}

\noindent
We derived $\mathbb{E}[||w_{t+1} - w'_{t+1}|| | i_t \neq j] \leq (1-\eta\gamma)\delta_t$ and $\mathbb{E}[||w_{t+1} - w'_{t+1}|| | i_t = j] \leq \delta_t + 2\eta G$. Substituting this yields

\begin{equation}
\delta_{t+1} \leq (1-\frac{1}{n})(1-\eta\gamma)\delta_t + (\frac{1}{n})(\delta_t + 2\eta G).
\end{equation}

\noindent
Expansion of the right-hand side yields

\begin{equation}
\delta_{t+1} \leq (1 - \eta\gamma - \frac{1}{n} - \frac{\eta \gamma}{n})\delta_t + (\frac{1}{n})(\delta_t + 2\eta G),
\end{equation}

\begin{equation}
\delta_{t+1} \leq (1 - \eta\gamma - \frac{1}{n} - \frac{\eta \gamma}{n})\delta_t + \frac{1}{n}\delta_t + \frac{2\eta G}{n},
\end{equation}

\begin{equation}
\delta_{t+1} \leq \delta_t(1 - n\gamma - \frac{1}{n} - \frac{\eta\gamma}{n} + \frac{1}{n}) + \frac{2\eta G}{n},
\end{equation}

\begin{equation}
\delta_{t+1} \leq (1 - \frac{\gamma}{\eta}(\frac{n-1}{n}))\delta_t + \frac{2\eta G}{n}.
\end{equation}

\noindent
This inequality expresses a linear recurrence of the form $\delta_t \leq a\delta_t + b$ with $\delta_0 = 0$, where $a = 1 - \frac{\gamma}{\eta}(\frac{n-1}{n})$ and $b = \frac{2\eta G}{n}$. Expressed as a summation for $\delta_{T}$, this linear recurrence is

\begin{equation}
\delta_{T} \leq a^T\delta_0 + b \sum_{k=0}^{T-1}a^k.
\end{equation}

\noindent
$\sum_{k=0}^{T-1}a^k$ is a geometric series, so $\sum_{k=0}^{T-1}a^k = \frac{1-a^T}{1-a}$. Moreover, since $0 < a < 1$, $a^T \rightarrow 0$ as $T \rightarrow \infty$. This implies

\begin{equation}
\delta_T \leq b(\frac{1}{1-a}),
\end{equation}

\noindent
which yields

\begin{equation}
\delta_T \leq \frac{2 \eta G}{n}(\frac{1}{\eta \gamma}) = \frac{2G}{\gamma n}.
\end{equation}

\noindent
Since the loss function is $G$-Lipschitz, we know $|f(w) - f(w')| \leq G||w - w'||$. Thus, $|f(w)_T - f(w')_T| \leq G\delta_T \leq G(\frac{2G}{\gamma n}) = \frac{2G^2}{\gamma n}$. Thus, $\delta_T \leq \frac{2G^2}{\gamma n}$. Moreover, since $\delta_T \leq \epsilon = O(\frac{1}{n})$, it follows that SGD under $\gamma$-strong convexity and $\beta$-smoothness, with constant $\eta \leq \frac{\gamma}{\beta^2}$ ran over $T$ iterations, is uniformly stable. 
\end{proof}

\bibliography{nag_stability}

@article{Pol64,
    author = {Polyak, Boris},
    title = {Some methods of speeding up the convergence of iteration methods},
    journal = {USSR Computational Mathematics and Mathematical Physics},
    volume = {4},
    pages = {1-17},
    year = {1964}
}

@article{Nes83,
    author = {Yurii Nesterov},
    title = {A Method of Solving a Convex Programming Problem with Convergence Rate $\mathcal{O}(\frac{1}{K^2})$},
    journal = {Soviet Mathematics Doklady},
    volume = {27},
    year = {1983}
}

@article{BE02,
    author = {Olivier Bousquet AND Andre Elisseeff},
    title = {Stability and Generalization},
    journal = {Journal of Machine Learning Research},
    volume = {2},
    year = {2002}
}

@book{NW06,
    author = {Jorge Nocedal AND Stephen Wright},
    title = {Numerical Optimization},
    publisher = {Springer},
    year = {2006}
}

@article{Bub15,
    author = {Sebastian Bubeck},
    title = {Convex Optimization: Algorithms and Complexity},
    journal = {Foundations and Trends in Machine Learning},
    volume = {8},
    year = {2015}
}

@inproceedings{HRS16,
    author = {Moritz Hardt AND Benjamin Recht AND Yoram Singer},
    title = {Train faster, generalize better: Stability of stochastic gradient descent},
    booktitle = {Proceedings of the 33rd International Conference on Machine Learning},
    year = {2016}
}

@article{LRP16,
    author = {Laurent Lessard AND Benjamin Recht AND Andrew Packard},
    title = {Analysis and Design of Optimization Algorithms via Integral Quadratic Constraints},
    journal = {SIAM Journal on Optimization},
    volume={26},
    number={1},
    publisher={SIAM},
    year = {2016}
}

@unpublished{CJY18,
    author = {Yuansi Chen AND Chi Jin AND Bin Yu},
    title = {Stability and Convergence Trade-off of Iterative Optimization Algorithms},
    year = {2018},
    note = {arXiv:1804.01619}
}

@inproceedings{AK21,
    author = {Amit Attia AND Tomer Koren},
    title = {Algorithmic Instabilities of Accelerated Gradient Descent},
    booktitle = {35th Conference on Neural Information Processing Systems (NeurIPS 2021)},
    year = {2021}
}

@book{BN01,
    author = {Aharon Ben-Tal AND Arkadi Nemirovski},
    title = {Lectures on Modern Convex Optimization: Analysis, Algorithms, and Engineering Applications},
    publisher = {MPS-SIAM},
    year = {2001}
}

@article{Yak62,
    author = {Vladimir Yakubovich},
    title = {Solution of certain matrix inequalities in the stability theory of nonlinear control systems},
    journal = {Soviet Mathematics Doklady},
    volume = {143},
    year = {1962}
}

@article{BH77,
    author = {JB Baillon AND G Haddad},
    title = {Some Properties of Angle-Bounded and n-Cyclically Monotone Operators},
    journal = {Israel Journal of Mathematics},
    volume = {26},
    year = {1977}
}
\end{document}